\documentclass[10pt]{amsart}
\usepackage{latexsym}
\usepackage{amsfonts}

\usepackage[headings]{fullpage}

\usepackage{amsmath,amsthm,amssymb}

\author[I.~Kapovich]{Ilya Kapovich}

\address{\tt Department of Mathematics, University of Illinois at
  Urbana-Champaign, 1409 West Green Street, Urbana, IL 61801, USA} \email{\tt
  kapovich@math.uiuc.edu}

\author[M.~Lustig]{Martin Lustig}\address{\tt I2M,
Centre de Math\'ematiques et Informatique,
Aix-Marseille Universit\'e,
39, rue F.~Joliot Curie, 
13453 Marseille 13, France} \email{\tt Martin.Lustig@univ-amu.fr}

\title[Patterson-Sullivan currents  and
and the Lipschitz metric]{Patterson-Sullivan currents, generic stretching factors and the asymmetric Lipschitz metric for
Outer space}

\newtheorem{thm}{Theorem}[section] \newtheorem{lem}[thm]{Lemma}
\newtheorem{cor}[thm]{Corollary} 
\newtheorem{prop}[thm]{Proposition} \theoremstyle{definition}
\newtheorem{defn}[thm]{Definition}

 \newtheorem{rem}[thm]{Remark}
\newtheorem{exmp}[thm]{Example}
\newtheorem{propdfn}[thm]{Proposition-Definition}

 \newtheorem{prob}[thm]{Problem}

\def\strutdepth{\dp\strutbox}
\def \ss{\strut\vadjust{\kern-\strutdepth \sss}}
\def \sss{\vtop to \strutdepth{
\baselineskip\strutdepth\vss\llap{$\diamondsuit\;\;$}\null}}

\def\strutdepth{\dp\strutbox}
\def \sst{\strut\vadjust{\kern-\strutdepth \ssss}}
\def \ssss{\vtop to \strutdepth{
\baselineskip\strutdepth\vss\llap{$\spadesuit\;\;$}\null}}

\def\strutdepth{\dp\strutbox}
\def \ssh{\strut\vadjust{\kern-\strutdepth \sssh}}
\def \sssh{\vtop to \strutdepth{
\baselineskip\strutdepth\vss\llap{$\heartsuit\;\;$}\null}}

\newcommand{\R}{\mathbb R}
\newcommand{\Z}{\mathbb Z}

\def\epsilon{\varepsilon}
\def\phi{\varphi}

\newcommand{\Curr}{\mbox{Curr}}
\newcommand{\Out}{\mbox{Out}}
\newcommand{\Aut}{\mbox{Aut}}

\newcommand{\FN}{F_N}   % F ou F_n ou F_N ?
\newcommand{\cvn}{\mbox{cv}_N}
\newcommand{\cvv}{\mbox{cv}_N^1}
\newcommand{\cvve}{\mbox{cv}_{N,\epsilon}^1}

\newcommand{\cvnbar}{\overline{\mbox{cv}}_N}
\newcommand{\CVN}{\mbox{CV}_N}
\newcommand{\CVNbar}{\overline{\mbox{CV}}_N}

\begin{document}

\begin{abstract}
We quantitatively relate the Patterson-Sullivant currents and generic stretching factors for free group automorphisms
to the asymmetric Lipschitz metric on Outer space and to Guirardel's intersection number.
\end{abstract}

\thanks{The first author was supported by the Collaboration Grant no. 279836 (2013-1018) from the Simons Foundation and by the NSF grant  DMS-1405146. Both authors  acknowledge support from U.S. National Science Foundation grants DMS 1107452, 1107263, 1107367 ``RNMS: GEometric structures And Representation varieties'' (the GEAR Network).}


\maketitle

%\tableofcontents

\section{Introduction}

For an integer $N\ge 2$ the \emph{unprojectivized Outer space} $\cvn$ is the set of all $\R$-trees equipped with a free discrete minimal isometric action of $F_N$, considered up to an $F_N$-equivariant isometry. We denote by $\cvv$ the set of all $T\in \cvn$ such that that the metric graph $T/F_N$ has volume $1$. 
The closure $\cvnbar$ of $\cvn$ with respect to the equivariant Gromov-Hausdorff convergence topology (or equivalently~\cite{Pau}, with respect to the hyperbolic length function topology) consists of all \emph{very small} minimal isometric actions of $F_N$ on $\R$-trees, again up to an $F_N$-equivariant isometry. There is a natural action of $\R_{>0}$ on $\cvnbar$ my multiplying the metric on a tree by a positive scalar. The subset $\cvn$ of $\cvnbar$ is invariant under this action, and the quotient $\CVN=\cvn/\R_{>0}$ is the \emph{projectivized Outer space}, originally introduced by Culler and Vogtmann in \cite{CV}. The quotient 
%$\CVNbar =\cvn/\R_{>0}$ 
$\CVNbar =\cvnbar/\R_{>0}$ 
is compact, and is called the \emph{Thurston compactification} of $\CVN$.
All of the above spaces admit natural $\Out(F_N)$-actions. The space $\CVN$ is naturally $\Out(F_N)$-equivariantly homeomorphic to $\cvv$, but it is useful to remember that technically $\cvv$ and $\CVN$ are distinct objects.

There are three main quantitative tools for studying points of $\cvnbar$. The first is the so-called ``asymmetric Lipschitz distance''.  If $T\in \cvn$ and $S\in \cvnbar$, the \emph{extremal Lipschitz distortion} is given by
$\displaystyle\Lambda(T,S):=\sup_{w\in F_N \smallsetminus \{1\}} \frac{||w||_S}{||w||_T}$.  It is known (see \cite{FM11} for details) that this $\sup$ is actually a $\max$, and that $\Lambda(T,S)$ is the infimum of the Lipschitz constants of all the $F_N$-equivariant Lipschitz maps $T\to S$.  It is also known that for all $T,S\in \cvv$ we have $\Lambda(T,S)\ge 1$, and that the equality holds if and only if $T=S$. The \emph{asymmetric Lipschitz distance} is defined as $d_L(T,S):=\log \Lambda(T,S)$ where $T,S\in\cvv$.  Although it is usually the case that $d_L(T,S)\ne d_L(S,T)$, the asymmetric distance $d_L$ satisfies all the other properties of being a metric, and it is known that the topology defined by $d_L$ on $\cvv$ coincides with the standard subspace topology for $\cvv\subseteq \cvn$.  Moreover, for any $T,S\in \cvv$ there exists an (in general non-unique)
$d_L$-geodesic path
from $T$ to $S$ in $\cvv$, given by natural ``folding lines''~\cite{FM11}. The asymmetric distance $d_L$ is a useful tool in the study of the geometry of $\Out(F_N)$ and it has found significant recent applications, see, for example, ~\cite{AK,AK2,AKB,Be11,FM11,FM12,LSV,Wh}.

Another two important quantitative tools for studying Outer space are two notions of a ``geometric intersection number''. The first of these was introduced by Guirardel in~\cite{Gui05} in the general setting of groups acting by isometries on $\R$-trees.
Guirardel's intersection number $i(T,S)$ (where $T,S\in\cvnbar$) is defined as the co-volume of the ``core'' for the action of $F_N$ on $T\times S$.  Guirardel's intersection number is symmetric and $\Out(F_N)$-invariant, and for $T,S\in\cvn$ one always has $0\le i(T,S)<\infty$.  However, for trees in $\partial \cvn=\cvnbar\smallsetminus\cvn$ it is often the case that $i(T,S)=\infty$ and $i(\cdot,\cdot)$ is discontinuous when viewed as a function on $\cvnbar\times\cvnbar$. Still, Guirardel's intersection number is a highly useful tool when studying the asymptotic geometry of $\cvn$ itself, particularly when looking at orbits of subgroups of $\Out(F_N)$ in $\cvv$ and $\cvn$. Examples of such applications can be found in~\cite{BBC,CMP,CP10,CP12,Gui05,Hor}.

The second notion of a ``geometric intersection number'' was introduced by Kapovich and Lustig in~\cite{KL2}. They constructed a \emph{geometric intersection form} $\langle \cdot, \cdot \rangle:\cvnbar \times \Curr(F_N)\to \R_{\ge 0}$,  where $\Curr(F_N)$ is the space of \emph{geodesic currents} on $F_N$.  See Section~\ref{Sect:curr} below and~\cite{Ka1,Ka2,KL1,KL2} for the more information and the background on geodesic currents. The geometric intersection form is continuous, $\Out(F_N)$-equivariant, and, importantly, it always gives a finite output, that is, for every $T\in\cvnbar$ and $\mu\in \Curr(F_N)$ one has $0\le \langle T,\mu\rangle<\infty$.  If $T\in \cvnbar$ and 
$g\in F_N \smallsetminus \{1\}$
then $\langle T,\eta_g\rangle=||g||_T$, where $\eta_g\in \Curr(F_N)$ is the ``counting current'' associated with $g$. By its very definition, $\langle \cdot,\cdot\rangle$ is an asymmetric gadget. However, its good properties, including finiteness and global continuity on $\cvnbar$, make the geometric intersection form a useful tool that has also found a number of significant applications to the study of the dynamics and geometry of $\Out(F_N)$. See, for example, ~\cite{BF08,BR12,CK,CP12a,CH13,CHL3,Ha09,Ha12,KL2,KL3,KL4,MR13,Rey12}.

For $\epsilon\ge 0$ we denote by $\cvve$ the set of all $T\in\cvv$ such that the length of the shortest simple closed loop in $T/F_N$ is $\ge \epsilon$. The set $\cvve$ is called the \emph{$\epsilon$-thick} part of $\cvv$.  Horbez~\cite{Hor} showed that, for any fixed $\epsilon>0$, if $T,S\in\cvve$,  one has 
\[
\frac{1}{K_1} \log  i_c(T,S)-K_2 \le d_L(T,S) \le K_1 \log  i_c(T,S) +K_2\tag{$\ddag$}
\]
for some constants $K_1\ge 1$, $K_2\ge 0$ depending only on $N$ and
$\epsilon$. Here $i_c(T,S)$ is the combinatorial version of
Guirardel's intersection number, where $i_c(T,S)$ is defined as the
number of 2-cells in $Core(T\times S)/F_N$, while
$i(T,S)$ is defined as the sum of the areas of all the 2-cells in
$Core(T\times S)/F_N$. Thus if, for $S,T\in\cvv$ the trees $T_0,S_0\in\cvn$ are obtained from $T$
and $S$ by making all edges have length $1$, then $i_c(T,S):=i(T_0,S_0)$.
Also, following the usual convention, in $(\ddag)$ we interpret $\log 0$ as $\log 0=0$.

In the present paper, for $T,S\in \cvve$ we relate $\Lambda(T,S)$ to a natural quantity defined in terms of $\langle \cdot, \cdot \rangle$. Via Horbez' result, this connection also relates the geometric intersection form  $\langle \cdot , \cdot \rangle$ to Guirardel's geometric intersection number $i(\cdot,\cdot)$.  Following the results of Furman~\cite{Fur} in the general set-up of word-hyperbolic groups, in~\cite{KN} Kapovich and Nagnibeda associated to every $T\in \cvn$ its \emph{Patterson-Sullivan current}. In general, the Patterson-Sullivan current is naturally defined only up to a multiplication by a positive scalar. Normalizing by the geometric intersection number with $T$ provides a canonical choice. Thus for a tree $T\in \cvn$ we denote by $\mu_T\in\Curr(F_N)$ the \emph{Patterson-Sullivan current} associated to $T$, normalized so that $\langle T,\mu_T\rangle=1$.
We refer the reader to Section~\ref{Sect:PS} below and to ~\cite{Fur,KN,KN2} for the precise definitions and background information about the Patterson-Sullivan currents. A key result obtained by Kapovich and Nagnibeda in~\cite{KN} shows that the map $J_{PS}:\cvv\to\Curr(F_N), \, T\mapsto\mu_T$ is a continuous $\Out(F_N)$-equivariant embedding.

Our main result (c.f. Theorem~\ref{thm:PS} below) is:
\begin{thm}\label{thm:A}
Let $N\ge 2$ and $\epsilon>0$. Then there exist constants $0<\delta_1\le \delta_2$
such that for every $T\in\cvve$ and every $S\in\cvnbar$  we have:
\[
\delta_1\le \frac{\langle S,\mu_T\rangle}{\Lambda(T,S)}\le \delta_2
\]
Therefore there exists a constant $c=c(N,\epsilon)>0$ such that for every $T\in \cvve$ and $S\in\cvv$ we have:
\[
\left| \log \langle S,\mu_T\rangle  - d_L(T,S)\right|\le c
\]
\end{thm}

Using the result of Horbez~\cite{Hor} stated in $(\ddag)$ above, Theorem~\ref{thm:A} directly implies (using the notation introduced after $(\ddag)$):
\begin{cor}\label{cor:guir}
Let $N\ge 2$ and $\epsilon>0$.  Then there exist constants $C_1,C_2\ge 1$ such that for any $T,S\in \cvve$ we have
\[
\frac{1}{C_1}\log  i_c(T,S)-C_2 \le \log \langle S,\mu_T\rangle  \le C_1 \log i_c(T,S) +C_2.
\]
\end{cor}

The proof of Theorem~\ref{thm:A} relies on several results regarding geodesic currents, particularly the result of Kapovich and Lustig~\cite{KL2} about the 
continuity of the 
geometric intersection form on $\cvnbar\times \Curr(F_N)$, mentioned above, and the result of Kapovich and Nagnibeda~\cite{KN} that the Patterson-Sullivan map $\cvv\to\Curr(F_N)$, $T\mapsto \mu_T$, is a continuous $\Out(F_N)$-equivariant embedding.  The most crucial point in the argument uses a result of Kapovich and Lustig~\cite{KN} which characterizes the case $\langle S,\nu\rangle=0$, where $S\in\cvnbar$ and $\nu\in\Curr(F_N)$ are arbitrary.  In particular, this characterization implies that every current $\mu$ with full support (such as the Patterson-Sullivan current $\mu_T$ for $T\in\cvv$) is \emph{filling}, that is, satisfies $\langle S,\mu\rangle >0$ for every $S\in\cvnbar$.  Modulo the tools mentioned above, the proof of Theorem~\ref{thm:A} is not difficult (although the proof does require an extra trick exploiting the $\Out(F_N)$-equivariant nature of certain functions and some nice properties of $d_L$). Still, 
Theorem~\ref{thm:A} and its applications obtained here do provide a conceptual clarification regarding the quantitative relationships between the two notions of a geometric intersection number used in the study of $\Out(F_N)$, and about their relationship to the asymmetric Lipschitz distance.  

%Additionally, Theorem~\ref{thm:A} motivates the definition of a new notion of a continuous symmetric and $\Out(F_N)$-invariant intersection number $I:\cvv\times\cvv\to \R_{>0}$, where for $T,S\in\cvv$ we put $I(T,S):=\langle S,\mu_T\rangle\langle T,\mu_S\rangle$.  The function $I(\cdot,\cdot)$ was  originally suggested to us by Arnaud Hilion as it appears to be relevant for attempting to define an analogue of the Weil-Petersson metric on $\cvv$. We discuss $I(\cdot,\cdot)$ in more detail in Problem~\ref{prob:2} below.

One of our main motivations for this paper has been to  better understand the properties of ``generic stretching factors'' for free group automorphisms.

\begin{propdfn}\label{defn:gen}\cite{KKS}
For any free basis $A$ of $F_N$ and any $S\in \cvnbar$ there exists  a number $\lambda_A(S)\ge 0$ 
with the following properly.

For a.e. trajectory $\xi=y_1y_2\dots y_n\dots$ of the simple non-backtracking random walk on $F_N$ with respect to $A$ (that is, for a ``random'' geodesic ray $\xi=y_1y_2\dots y_n\dots$ over $A^{\pm 1}$ with $y_i\in A^{\pm 1}$) we have $||y_1y_2\dots y_n||_A=n+o(n)$ and
\[
\lim_{n\to\infty} \frac{||y_1y_2\dots y_n||_S}{n}=\lim_{n\to\infty} \frac{||y_1y_2\dots y_n||_S}{||y_1y_2\dots y_n||_A}=\lambda_A(S).
\]

The number $\lambda_A(S)$ is called~\cite{Ka2,KKS} the \emph{generic stretching factor} of $S$ with respect to $A$.
\end{propdfn}

The term ``non-backtracking'' in ``non-backtracking simple random walk'' refers to the fact that for this random walk, if $x,y\in A\cup A^{-1}$, the transition probability for $x$ to be followed by $y$ is equal to $1/(2N-1)$ if $y\ne x^{-1}$ and is equal to $0$ if $y=x^{-1}$. Thus the trajectories of this random walk are semi-infinite freely reduced words over $A^{\pm 1}$.
Informally, the generic stretching factor $\lambda_A(S)\ge 0$ captures the distortion $\frac{||y_1y_2\dots y_n||_S}{n}$ where $y_1\dots y_n$ is a ``random'' freely reduced word of length $n$ over $A$, as $n$ tends to infinity.
The existence of $\lambda_A(S)\ge 0$ follows from general ergodic-theoretic considerations, as observed in \cite{KKS}. As noted in Remark~\ref{rem:>0} below, one actually has $\lambda_A(S)>0$ for every $S\in\cvnbar$.

Let $A$ be a free basis of $F_N$ and consider the Cayley tree $T_A\in
\cvn$, with all edges of length $1/N$, so that $T_A\in \cvv$. Thus for
every $w\in F_N$ we have $||w||_A=N||w||_{T_A}$ where $||w||_A$ is the
cyclically reduced length of $w$ over $A^{\pm 1}$. It is known that the Patterson-Sullivan current $\mu_{T_A}$ is equal to the ``uniform current'' $\nu_A$ on $F_N$ corresponding to $A$. Using the interpretation of $\langle S,\nu_A\rangle$ as the ``generic stretching factor''  $\lambda_A(S)$ of $S\in \cvn$ with respect to $A$~\cite{Ka2}, as a consequence of Theorem~\ref{thm:A} we also obtain (see Theorem~\ref{thm:gen1'} below):

\begin{cor}\label{cor:gen1}
Let $N\ge 2$. There exists a constant $\delta=\delta(N)\in (0,1)$ with the following property:

For any free basis $A$ of $F_N$ and any $S\in \cvnbar$ we have
\[
0<\delta\le \frac{\lambda_A(S)}{\Lambda(T_A,S)}\le \frac{1}{N}.  \tag{$\dag$}
\]
\end{cor}

We are particularly interested in relationships between generic stretching factors and extremal stretching factors in the context of Cayley trees of $F_N$ and of elements of $\Out(F_N)$. Note that if $A$ is a free basis of $A$ then $NT_A\in\cvn$ is the standard Cayley graph of $F_N$ with respect to $A$, where all edges have length $1$. 

If $\phi\in\Out(F_N)$ and $w\in F_N$, then, since $\phi$ is an outer automorphism, it acts on the conjugacy classes of elements of $F_N$ (rather than on elements of $F_N$).
By convention, for $\phi\in \Out(F_N)$ and $w\in F_N$, if $\phi(w)$ appears in an expression that depends only on the conjugacy class $\phi([w])$,  we will use $\phi(w)$ to mean any representative of that conjugacy class. 

%If $\phi\in\Out(F_N)$ and $w\in F_N$, then, since $\phi$ is an outer automorphism, the element $\phi(w)\in F_N$ is  only well-defined up-to conjugation in $F_N$. However, if $T\in \cvnbar$, the translation length $||g||_T$ only depends on the conjugacy class of $g$ in $F_N$. Therefore for $\phi\in\Out(F_N)$, $w\in F_N$ and $T\in\cvnbar$ the quantity $||\phi(w)||_T$ is well defined. By convention, in such cases we will interpret $||\phi(w)||_T$ to mean $||\Phi(w)||_T$ where $\Phi\in\Aut(F_N)$ is any representative of $\phi$ in $\Out(F_N)$.

\begin{defn}[Extremal and generic stretching factors of automorphisms]
\label{gen-stretch}
Let $A$ be a free basis of $F_N$ and let $\phi\in\Out(F_N)$.

Denote
\[
\Lambda_A(\phi):=\Lambda(T_A, T_A\phi)=\sup_{w\ne 1} \frac{||\phi(w)||_A}{||w||_A}=e^{d_L(T_A,T_A\phi)}
\]
and refer to $\Lambda_A(\phi)$ as the \emph{extremal stretching factor} for $\phi$ with respect to $A$.

Also, denote 
$\lambda_A(\phi):=\lambda_A(NT_A\phi) = N\lambda_A(T_A\phi)$.

Thus for a.e.   trajectory $\xi=y_1\dots y_n\dots$ of the simple non-backtracking random walk on $F_N$ with respect to $A$ we have
\[
\lambda_A(\phi)=\lim_{n\to\infty} \frac{||\phi(y_1y_2\dots y_n)||_A}{n}=\lim_{n\to\infty} \frac{||\phi(y_1y_2\dots y_n)||_A}{||y_1y_2\dots y_n||_A}.
\]
We call $\lambda_A(\phi)$ the \emph{generic stretching factor} of $\phi$ with respect to $A$.
\end{defn}

Thus $\Lambda_A(\phi)$ measures the maximal distortion $\frac{||\phi(w)||_A}{||w||_A}$ 
as $w$ varies over all non-trivial elements of
$F_N$, while $\lambda_A(\phi)$ captures the ``generic distortion''
$\frac{||\phi(w)||_A}{||w||_A}$, 
where $w$ is a ``long random'' freely reduced (or cyclically reduced) word over $A^{\pm 1}$.  In practice, $\Lambda_A(\phi)$ is easy to compute since it is known (see, e.g. \cite{FM11}) that $\Lambda_A(\phi)=\max_{1\le ||w||\le 2} \frac{||\phi(w)||_A}{||w||_A}$.

The generic stretching factors
$\lambda_A(\phi)$ were introduced in \cite{KKS} and further studied in
\cite{Fra,Ka2,KL3,Sharp}. In particular, it is proved in \cite{KKS} that for every $\phi\in\Out(F_N)$ the number $\lambda_A(\phi)$ is rational and moreover, $2N\lambda_A(\phi)\in \mathbb Z[\frac{1}{2N-1}]$ and that there exists an algorithm that, given $\phi$, computes $\lambda_A(\phi)$.  The definitions directly imply that  $\lambda_A(\phi) \leq \Lambda_A(\phi)$.
However, other than this fact, the quantitative relationship between $\Lambda_A(\phi)$ and $\lambda_A(\phi)$ remained unclear.

Let $N\ge 2$ and $F_N=F(a_1,\dots, a_N)$ with $A=\{a_1,\dots, a_N\}$. 
Define 
\[
\rho_N:=\inf_{\phi\in \Out(F_N)} \frac{\lambda_A(\phi)}{\Lambda_A(\phi)}.
\]

Since for every $\phi\in \Out(F_N)$ we have $T_A, T_A\phi\in\cvve$ with $\epsilon=\frac{1}{N}$, Corollary~\ref{cor:gen1} directly implies: 
\begin{thm}\label{cor:gen2}
For every $N\ge 2$ we have $\rho_N>0$.
\end{thm} 

Therefore for every $\phi\in \Out(F_N)$ we have
\[
0<\rho_N\le \frac{\lambda_A(\phi)}{\Lambda_A(\phi)}\le  1.
\]

Our proof that $\rho_N>0$ does not give any
explicit quantitative information about $\rho_N$. It would be
interesting to find some explicit bounds from above and below for $\rho_N$, and perhaps to even compute $\rho_N$,
at least for small values of $N$. We show in Proposition~\ref{prop:limit}  that $\displaystyle \lim_{N\to\infty} \rho_N=0$ and that $\rho_N=O(\frac{1}{N})$.

As another application, we obtain (c.f. Corollary~\ref{cor:gen3'} below):
\begin{cor}\label{cor:gen3}
Let $N\ge 2$ and $F_N=F(a_1,\dots, a_n)$ with $A=\{a_1,\dots, a_N\}$.
There exists $D=D(N)\ge 1$ such that for every $\phi\in \Out(F_N)$ we
have
\[
\frac{1}{D} \log \lambda_A(\phi) \le \log \lambda_A(\phi^{-1})\le  D
\log \lambda_A(\phi) .
\]
\end{cor}

Let $\phi\in \Out(F_N)$. Recall that the 
\emph{algebraic stretching factor} $\lambda(\phi)$ 
%
%\marginpar{\tiny ML:  Maybe one should recall here the original paper of R. Bowen who (I believe) introduced this notion under the name of "growth rate" and used it to show that it is always a lower bound to the topological entropy of any self-map $f: X \to X$ which realizes the given auto on $\pi_1(X)$. Unfortunately I don't have the precise reference ready ....}
%
is defined as
\[
\lambda(\phi):=\sup_{w\in F_N, w\ne 1}\lim_{n\to\infty} \sqrt[n]{||\phi^n(w)||_S}\]
where $S\in \cvn$ is an arbitrary base point. It is known that the
limit in the last 
equality always exists,  
that this
definition of $\lambda(\phi)$ does not depend on the choice of $S\in\cvn$, and that we always have $\lambda(\phi)\ge 1$.   An element $\phi\in\Out(F_N)$ is called \emph{exponentially growing} if $\lambda(\phi)>1$, and \emph{polynomially growing} if $\lambda(\phi)=1$.   Indeed, it 
is known (see for example \cite{Levitt}), that $\phi$ is polynomially growing if and only if for every $w\in F_N$ and $S\in\cvn$ the sequence $||\phi^n(w)||_S$ is bounded above by a polynomial in $n$. 

The algebraic stretching factor $\lambda(\phi)$ can be read-off from any relative train-track representative $f: \Gamma \to \Gamma$ of $\phi$ 
as the maximum of the Perron-Frobenius eigenvalues for any of the canonical irreducible diagonal blocks of the (non-negative) transition matrix $M(f)$.

As another application of the results of this paper, we explain how the generic stretching factor
$\lambda_A(\phi^n)$ grows in terms of $n$ for an arbitrary
$\phi\in \Out(F_N)$. Thus we obtain (c.f. Theorem~\ref{thm:p'} below) the following result, which answers Problem~9.2 posed in \cite{KKS}:

\begin{thm}\label{thm:p}
Let $A$ be a free basis of $F_N$ and let $\phi\in\Out(F_N)$ and let $
%\lambda=
\lambda(\phi)$ be the algebraic stretching factor of $\phi$. Then there exist constants $c_1,c_2>0$ and an integer $m\ge 0$ such that for every $n\ge 1$ we have
%\marginpar{\tiny ML:  I took out the parentheses around $\lambda(\phi)$ in the formula.}
\[
c_1 \, \lambda(\phi)^n \, n^m\le \lambda_A(\phi^n) \le c_2  \, \lambda(\phi)^n \, n^m.
\]
Moreover, if $\phi$ admits an expanding train-track representative with
an irreducible transition matrix (e.g. if $\phi$ is fully
irreducible), then $m=0$ and $\lambda(\phi)>1$.
\end{thm} 

The ``polynomial growth degree'' $m$ in this result is bounded above by the number of strata of any relative train track representative $f$ as above which have PF-eigenvalue equal to $\lambda$, and it has been determined precisely by Levitt in \cite{Levitt}, see the proof of Proposition \ref{prop:power} below.

\smallskip

%\marginpar{\tiny ML:  This (or right after the Acknowledgements) could be a suitable place for a short ``Glossary'' where we juste (re)state the definitions of all the numerical terms used in the paper.}

{\bf Acknowledgements:}  We thank Matt Clay and Camille Horbez for useful discussions about Guirardel's intersection number.  We are also grateful to Brian Ray and Paul Schupp for conducting helpful computer experiments  with generic stretching factors of free group automorphisms.

\section{Preliminaries}

\subsection{Basic terminology and notations related to Outer space}
\label{basic-outer}

We denote by $\cvn$ the unprojectivized Outer space, that is the space of all free discrete minimal isometric actions of $F_N$ on $\mathbb R$-trees, considered up to $F_N$-equivariant isometry.
Denote by $\cvnbar$ the closure of $\cvn$ in the equivariant Gromov-Hausdorff convergence topology (or, equivalently, in the hyperbolic length functions topology). It is known~\cite{BF93,CL,Gui98} that $\cvnbar$ consists of all the \emph{very small} non-trivial minimal isometric actions of $F_N$ on $\mathbb R$-trees, again considered up to $F_N$-equivariant isometry.
Recall that a point $T\in\cvnbar$ is uniquely determined by its \emph{translation length function} $||\cdot||_T:F_N\to [0,\infty)$, where for $w\in F_N$ we have $||w||_T=\inf_{x\in T} d(x,wx)=\min_{x\in T} d(x,wx)$.

The space $\cvnbar$ has a natural right $\Out(F_N)$-action, where for $w\in F_N$ and $T\in\cvnbar$ we have $||w||_{T\phi}=||\phi(w)||_T$. 
It is sometimes useful to convert this action 
to a left $\Out(F_N)$-action by setting $\phi T:=T\phi^{-1}$. 
Denote $\cvv:=\{T\in\cvn \mid vol(T/F_N)=1\}$ and refer to $\cvv$ as the \emph{volume-normalized Outer space} or just \emph{normalized Outer space}.
Then $\cvn$ is  an open dense $\Out(F_N)$-invariant subset of $\cvnbar$, and $\cvv$ is a closed $\Out(F_N)$-invariant subset of $\cvn$ (but of course $\cvv$ is not closed in $\cvnbar$). 

There is a natural action of $\mathbb R_{>0}$ on $\cvn$ and $\cvnbar$ by scalar multiplication, which yields the corresponding \emph{projectivizations} $\CVN=\cvn/\mathbb R_{>0}$ and $\CVNbar=\cvnbar/\mathbb R_{>0}$.  For a tree $T\in\cvnbar$ we denote its projective class in $\CVNbar$ by $[T]$. Thus $[T]=\{cT \mid c>0\}$.
Note that $\CVN$ is canonically $\Out(F_N)$ equivariantly homeomorphic to $\cvv$, but it is still important to remember that technically $\CVN$ and $\cvv$ are distinct objects.

For $\epsilon>0$ we denote by $\cvve$ the set of all  $T\in\cvv$ such that 
the shortest non-trivial immersed circuit in the metric graph $T/F_N$ has length
$\ge \epsilon$. Equivalently, $\cvve$ is the set of all $T\in\cvv$ such that for every $w\in F_N\smallsetminus\{1\}$ we have $||w||_T\ge \epsilon$. For every $\epsilon>0$ the set $\cvve \subseteq \cvv$ is a closed $\Out(F_N)$-invariant subspace, and the quotient $\cvve/\Out(F_N)$ is compact.

A 
\emph{chart} 
on $F_N$ is an isomorphism $\alpha: F_N\to\pi_1(\Gamma,p)$ where $\Gamma$ is a finite connected graph with all vertices of degree $\ge 3$ and where $p$ is a base vertex in $\Gamma$ (which is usually suppressed).  Every such $\alpha$ defines an open cone in $\cvn$ consisting of assigning arbitrary positive lengths to edges of $\Gamma$ and then lifting this assignment to the universal cover $\widetilde \Gamma$ to get an element $T\in\cvn$. The intersection of such an open cone with $\cvv$ is an open simplex 
$\Delta$ 
in $\cvv$ of dimension $m-1$, where $m$ is the number of 
unoriented 
edges of $\Gamma$. Every point $T\in \cvn$ belongs to a unique open cone of this form, and every point of $\cvv$ belongs to a unique 
%open simplex. 
such open simplex $\Delta$.

The space $\CVNbar$ is known to be compact and finite-dimensional.

\subsection{Asymmetric Lipschitz distance}

For points $T\in\cvn$ and $S\in\cvnbar$ denote 
\[
\Lambda(T,S)=\sup_{w\in F_N\smallsetminus\{1\}} \frac{||w||_S}{||w||_T}.
\]

If $T,S\in\cvv$, we also denote $d_L(T,S):=\log \Lambda(T,S)$. As noted in the Introduction, for $T,S\in \cvv$, the quantity $d_L(T,S)$ is often called the \emph{asymmetric Lipschitz distance} from $T$ to $S$.

\begin{rem}\label{rem:candidates}
If  if  $T\in\cvn$ and $S\in\cvnbar$ then
$0<\Lambda(T,S)<\infty$. Moreover, it is known~\cite{FM11,Wh} that for any open simplex 
$\Delta \subset \cvn^1$ as in subsection \ref{basic-outer}
there exists a
finite subset $C_\Delta\subseteq F_N\smallsetminus\{1\}$ such that for every
$T\in\Delta$ and every
$S\in\cvnbar$ we have \[\Lambda(T,S)=\max_{w\in C_\Delta}
\frac{||w||_S}{||w||_T}.\]
The set $C_\Delta$ can be chosen to be contained in the subset of all elements which are represented by paths that cross at most twice over every non-oriented edge of $\Gamma = T/\FN$, for $T \in \Delta$.

Note also that from the definition we see that for every
$T\in\cvn$, $S\in\cvnbar$ and $\phi\in \Out(F_N)$ one has
$\Lambda(T,S)=\Lambda(\phi T,\phi S)$.
\end{rem}

\subsection{Geodesic currents}\label{Sect:curr}

We refer the reader to \cite{Ka2,KL1,KL2,KL3} for detailed background on geodesic currents, and we only recall a few basic definitions and facts here.
Let $\partial^2 F_N=\partial F_N\times \partial F_N\smallsetminus diag$, and endow $\partial^2 F_N$ with the subspace topology and with the diagonal $F_N$-action by translations.
A \emph{geodesic current} on $F_N$ is a positive Borel measure $\mu$ on $\partial^2 F_N$ such that $\mu$ is finite on compact subsets, $F_N$-invariant and ``flip''-invariant (where the ``flip'' map $\partial^2 F_N\to\partial^2 F_N$ interchanges the two coordinates). The space of all geodesic currents on $F_N$ is denoted $\Curr(F_N)$. The space $\Curr(F_N)$ comes equipped with a natural weak*-topology and a natural left $\Out(F_N)$-action by affine homeomorphisms. 

Let $\alpha: F_N\to\pi_1(\Gamma,p)$ be a chart on $F_N$, and consider $\widetilde \Gamma$ with the simplicial metric where every edge has length 1. Then there is a natural $F_N$-equivariant quasi-isometry (given for any point $p\in \widetilde\Gamma$ by the orbit map $F_N\to \widetilde\Gamma$, $g\mapsto gp$) 
between $F_N$ and $\widetilde \Gamma$, which induces a canonical $F_N$-equivariant homeomorphism between $\partial F_N$ and $\partial \widetilde \Gamma$.
We will therefore identify $\partial F_N$ with $\partial \widetilde \Gamma$ using this homeomorphism 
without invoking it explicitly, whenever it is convenient.

A non-degenerate geodesic segment $\gamma$ is $\widetilde\Gamma$ defines a \emph{cylinder set} $Cyl_\alpha(\gamma)$ consisting of all $(X,Y)\in \partial^2 F_N$ such that the geodesic from $X$ to $Y$ in $\widetilde\Gamma$ passes through $\gamma$ (in the correct direction). The sets $Cyl_\alpha(\gamma)$, as $\gamma$ varies among all non degenerate geodesic edge-paths in $\widetilde\Gamma$, are compact and open, and form a basis for the topology on $\partial^2 F_N$. Note that for $w\in F_N$ we have $Cyl_\alpha(w\gamma)=wCyl_\alpha(\gamma)$.  If $\mu\in \Curr(F_N)$ and $v$ is a non-degenerate reduced edge-path in $\Gamma$, we define the \emph{weight} $\langle v,\mu\rangle_\alpha:=\mu(Cyl_\alpha(\gamma))$ where $\gamma$ is any lift of $v$. Since the measure $\mu$ is $F_N$-invariant, this definition does not depend on the specific choice of the lift $\gamma$ of $v$ to $\widetilde \Gamma$.  A current $\mu$ is uniquely determined by its collection of weights with respect to a given chart. Moreover, if $\mu_n,\mu\in \Curr(F_N)$ and $\alpha$ is a chart as above, then $\lim_{n\to\infty} \mu_n=\mu$ in $\Curr(F_N)$ if and only if for every non-degenerate reduced edge-path $v$ in $\Gamma$ we have $\lim_{n\to\infty} \langle v,\mu_n\rangle_\alpha=\langle v,\mu\rangle_\alpha$. 

For every $w\in F_N \smallsetminus \{1\}$  there is an associated \emph{counting current} $\eta_w\in \Curr(F_N)$, which depends only on the conjugacy class $[w]$ of $w$ in $F_N$ and satisfies $\eta_{w^{-1}}=\eta_w$ and $\eta_{w^n}=n\,\eta_w$ for all integers $n\ge 1$, and such that $\phi\,\eta_w=\eta_{\phi(w)}$ for all $\phi\in\Out(F_N)$, $w\in F_N\smallsetminus \{ 1\}$. 
The precise definition of $\eta_w$ is not important at the moment, but we will recall some of its basic properties later, as necessary.
The set $\{c\,\eta_w\mid c>0, w\in F_N, w\ne 1\}$ of the so-called \emph{rational currents} is dense in $\Curr(F_N)$.

Be aware that in general for a representative (even a train-track representative) $f: \Gamma \to \Gamma$ of $\phi$ one has $\langle v, \phi\mu\rangle_\alpha \neq \langle [f(v)],\mu\rangle_\alpha$, where $[f(v)]$ denotes the edge-path obtained from $f(v)$ by reduction (= iterative contraction of any backtracking path).

\subsection{Intersection form}

In~\cite{KL2} Kapovich and Lustig proved the existence of a continuous \emph{geometric intersection form} between points of $\cvnbar$ and geodesic currents:

\begin{prop}\cite{KL2}\label{prop:iform} 
There exists a unique continuos function $\langle\cdot ,\cdot \rangle :
\cvnbar
\times \Curr(F_N)\to [0,\infty)$, 
called the  \emph{geometric intersection form}, with the following properties:
\begin{enumerate}
\item For any $\mu_1,\mu_2\in \Curr(F_N)$, $T\in\cvnbar$, $c_1,c_2\ge 0$ and $r>0$ we have:
\[
\langle r T, c_1\mu_1+c_2\mu_2\rangle=rc_1\langle T,\mu_1\rangle+rc_2\langle T,\mu_2\rangle
\]
\item For any $T\in
\cvnbar
$, $\mu\in\Curr(F_N)$ and $\phi\in\Out(F_N)$ we have:
\[
\langle \phi T, \phi \mu\rangle=\langle T,\mu\rangle
\]
\item For any $T\in\cvnbar$ and 
$w\in F_N \smallsetminus \{1\}$ 
we have:
\[
\langle T,\eta_w\rangle=||w||_T
\]
\item 
For any $T\in\cvn$ (with the associated chart $\alpha: \FN \to 
\pi_1(T/\FN)
$) and any $\mu\in\Curr(F_N)$ we have:
\[
\langle T, \mu\rangle= \sum_{e  \in {\rm Edges}(T/\FN)}\frac{1}{2}\langle e;\mu\rangle_\alpha
\]
where the summation is taken over all oriented edges of the graph $T/\FN$. 
\end{enumerate} 
\end{prop}

\section{Tree-current morphisms and extremal Lipschitz
  distortion}
\label{Lip-distortion}

Recall that a current $\mu\in \Curr(F_N)$ is called \emph{filling} if for every $S\in \cvnbar$ we have $\langle S,\mu\rangle >0$.

We proved in \cite{KL3} that for a current $\mu\in\Curr(F_N)$ and a tree $T\in\cvnbar$ we have $\langle T,\mu\rangle=0$ if and only if the support of $\mu$ is contained in the ``dual algebraic lamination'' of $T$ (in the sense of \cite{CHL2}). Using this fact, it was shown in \cite{KL3}  that if $\mu$ is a current with full support, then $\mu$ is filling. We denote by $\Curr_{fill}(F_N)$ the set of all filling $\mu\in \Curr(F_N)$, and endow  $\Curr_{fill}(F_N)$ with the subspace topology given by the inclusion  $\Curr_{fill}(F_N)\subseteq \Curr(F_N)$.

\begin{defn}[Tree-current morphism]
A \emph{tree-current morphism} is a continuous
function $J:\cvn^1\to \Curr(F_N)$ such that for every $T\in \cvn^1$ and
every $\phi\in\Out(F_N)$ we have $J(\phi T)=\phi\ J(T)$.

A \emph{filling tree-current morphism} is a tree-current morphism $J:\cvn^1\to \Curr(F_N)$  such that for every $T\in\cvv$ the current $J(T)\in\Curr(F_N)$ is filling.
\end{defn}

\begin{lem}\label{lem:contin}
The function $\cvv\times \cvnbar\to \mathbb R$, $(T,S)\mapsto \Lambda(T,S)$, is continuous.
\end{lem}
\begin{proof}
Let $T\in \cvv$ be arbitrary. 

Let $\Delta_1,\dots, \Delta_m$ be all the open simplicies in $\cvv$
whose closures in $\cvv$ contain $T$. 

Set
$C_T=\cup_{i=1}^mC_{\Delta_i}$. Note that $U=\Delta_1\cup \dots
\cup\Delta_m$ is a neighborhood of $T$ in $\cvv$.

Thus for every $T'\in U$ and every $S\in \cvnbar$ we have
\[
\Lambda(T',S)=\max_{w\in C_T} \frac{||w||_S}{||w||_{T'}}.
\]
Therefore the function $\Lambda(T',S)$ is continuous on $U\times \cvnbar$. Since $T\in\cvv$ was arbitrary,
the conclusion of the lemma follows.
\end{proof}

Let $J$ be a filling tree-current morphism. Then for any $S\in \cvnbar$ and $c>0$ we have $\frac{\langle S, J(T)\rangle}{\Lambda(T,S)}=\frac{\langle cS, J(T)\rangle}{\Lambda(T,cS)}$.
Note also that since $J(T)$ is a filling current, for every $S\in\cvnbar$ we have  $\langle S, J(T)\rangle >0$. Therefore we have a well defined function
\[
f:\cvv\times \CVNbar\to (0,\infty)
\]
given by $f(T,[S])=\frac{\langle S, J(T)\rangle}{\Lambda(T,S)}$, where
$T\in \cvv$ and $S\in\cvnbar$.

\begin{lem}\label{lem:cont}
Let $J$ be a filling tree-current morphism. Then the function 
\[
f:\cvv\times \CVNbar\to (0,\infty), \quad (T,S)\mapsto \frac{\langle S, J(T)\rangle}{\Lambda(T,S)}
\]
is continuous.

\end{lem}

\begin{proof}

The conclusion of the lemma follows directly from Lemma~\ref{lem:contin} together with the continuity of the 
the geometric intersection form $\langle \cdot ,
\cdot \rangle$. 
\end{proof}

\begin{cor}\label{cor:minmax}
Let $K\subseteq \cvv$ be a compact
subset, and let
$J:\cvv\to \Curr_{fill}(F_N)$ be a filling tree-current morphism.

Then there exist $\delta_1=\delta_1(K,J)>0$ and $\delta_2=\delta_2(K,J)>0$ such that for every $T\in K$ and every
$S\in\cvnbar$ we have $\delta_1\le f(K,[S])\le \delta_2$.
\end{cor}

\begin{proof}
The set $K\times \CVNbar$ is a compact Hausdorff space and, by Lemma~\ref{lem:cont}, $f:K\times
\CVNbar\to (0,\infty)$ is a continuous function. Therefore $f$ achieves a
positive minimum $\delta_1$ and a positive maximum $\delta_2$ on $K\times \CVNbar$, and the conclusion of the
corollary follows.
\end{proof}

\begin{cor}\label{cor:main}
Let $K\subseteq \cvv$ be a compact
subset, let $\mathcal T_K=\cup_{\phi\in \Out(F_N)} \phi K$ and let
$J:\cvv\to \Curr(F_N)$ be a filling tree-current morphism.

Let furthermore $\delta_1=\delta_1(K,J)>0$ 
and 
$\delta_2=\delta_2(K,J)>0$ be the
constants provided by Corollary~\ref{cor:minmax}.

Then for every $T\in \mathcal T_K$ and every $[S]\in\CVNbar$ we have
\[
0< \delta_1 \le \frac{\langle S, J(T)\rangle}{\Lambda(T,S)} \le \delta_2<\infty.
\]

\end{cor}
\begin{proof}
Let  $T\in \mathcal T_K$ and $[S]\in\CVNbar$ be arbitrary.

Then there exist $T'\in K$ and $\phi\in\Out(F_N)$ such that $T=\phi T'$.  By $\phi$-equivariance of $J$ we have $J(T)=\phi J(T')$.
Denote $S'=\phi^{-1}S$, so that $\phi S'=S$. Then
\begin{gather*}
\frac{\langle S, J(T)\rangle}{\Lambda(T,S)}=\frac{\langle \phi S', \phi J(T')\rangle}{\Lambda(\phi T',\phi S')}  = \frac{\langle S', J(T')\rangle}{\Lambda(T', S')}=f(T',[S'])\in [\delta_1.\delta_2], 
\end{gather*}
where the last inclusion holds 
by 
Corollary~\ref{cor:minmax} since $T'\in K$.
\end{proof}

Note that Corollary~\ref{cor:main} does not require the tree-current morphism $J:\cvv\to \Curr_{fill}(F_N)$ to be injective, although in the specific applications of interest to us $J$ will be injective.

\section{Patterson-Sullivan currents and extremal Lipschitz distortion}\label{Sect:PS}

\subsection{Volume entropy and the Patterson-Sullivan currents}

We only give here a brief summary of basic definitions and facts regarding Patterson-Sullivan currents for points of $\cvn$.
We refer the reader to~\cite{Fur,Coor,Kaim,KN} for more detailed background information about Patterson-Sullivan measures and Patterson-Sullivan currents in the context of word-hyperbolic groups and Gromov-hyperbolic spaces.

Let $T\in\cvn$, where $N\ge 2$.  Since $F_N$ and $T$ are $F_N$-equivariantly quasi-isometric, there is a natural identification of $\partial F_N$ and $\partial T$, which we will use later on.

The \emph{volume entropy} $h(T)$ of $T$ is defined as
\[
h(T) 
:=\lim_{R\to\infty} \frac{\log(\#\{w\in F_N \mid d_T(p,wp)\le R\})}{R}
\]
where $p\in T$ is an arbitrary base point. It is known that the above definition does not depend on the choice of a base-point $p\in T$ and that we have $h(T)>0$ for every $T\in\cvn$.
It is also known that $h(T)$ is exactly the critical exponent of the \emph{Poincare series} 
\[
\Pi_p(s)=\sum_{w\in F_N} e^{-s d_T(p,wp)}.
\]
In other words,
$\Pi_p(s)$ converges for all $s>h(T)$ and diverges for all $s\le h(T)$.
It is also known that as $s\to h+$, any weak limit $\nu$ of the measures
\[
\frac{1}{\Pi_p(s)} \sum_{w\in F_N} e^{-s d_T(p,wp)} {\rm Dirac}(wp)
\]
is a probability measure supported on $\partial T=\partial F_N$.
Any such $\nu$ is called a \emph{Patterson-Sullivan} measure on $\partial F_N$ corresponding to $T$, and the measure class of $\nu$ is canonically determined by $T$.
As follows from general results of Furman~\cite{Fur}, in this case there exists a unique, up to a scalar multiple, geodesic current $\mu$ in the measure class of $\nu\times\nu$ on $\partial^2 F_N$.
We call the unique scalar multiple $\mu_T$ of $\mu$ such that $\langle T, \mu_T\rangle=1$, the \emph{Patterson-Sullivan current} for $T\in\cvn$. One also has that the current $\mu_T$ has full support (this follows, for example, both from the general results of Furman~\cite{Fur} and from the explicit  formulas for $\mu_T$ obtained in \cite{KN}). 

\begin{prop}\label{prop:KN}
The map 
$$J_{PS}: \cvv\to Curr(F_N), \, 
T \mapsto \mu_T$$ 
is a filling tree-current morphism.
\end{prop}

\begin{proof}
Since $\mu_T$ has full support, by a result of Kapovich and Lustig~\cite[Corollary 1.3]{KL3},  it follows that 
$\mu_T\in \Curr_{fill}(F_N)$.  The fact that $J_{PS}$ is a continuous $\Out(F_N)$-equivariant map was proved by Kapovich and Nagnibeda~\cite{KN}.
Thus $J_{PS}$ is indeed a filling tree-current morphism, as claimed.
\end{proof}

The fact that for $T\in\cvv$ the Patterson-Sullivan current $\mu_T$ is filling, 
i.e. that $\langle S,\mu_T\rangle\ne 0$ for every $S\in\cvnbar$, is quite non-trivial and does not follow directly from Proposition~\ref{prop:iform}. This fact, which requires a general result from \cite{KL3} characterizing the case where $\langle S,\mu\rangle=0$ (where $S\in\cvnbar$ and $\mu\in\Curr(F_N)$), is, in a sense, the place where the real ``magic'' 
in the proofs of the main results of the present paper happens.

We now obtain Theorem~\ref{thm:A} from the Introduction:

\begin{thm}\label{thm:PS}
Let  $N\ge 2$ and $\epsilon>0$. Then there exist constants $\delta_2\ge \delta_2>0$ 
such that for
every $T\in \cvve$, $S\in\cvnbar$ we have:
\[
\delta_1\le \frac{\langle S,\mu_T\rangle}{\Lambda(T,S)}\le \delta_2
\]
Therefore there exists a constant $c > 0$ such that for
every $T\in \cvve$ and $S\in\cvv$ we have:
\[
\left| \log \langle S,\mu_T\rangle  - d_L(T,S)\right|\le c.
\]
\end{thm}
\begin{proof}
Since $\cvve/\Out(F_N)$ is compact and the action of $\Out(F_N)$ on $\cvve$ is properly discontinuous,  there exists a compact subset $K\subseteq \cvve$ such that $\cvve=\mathcal T_K=\cup_{\phi\in\Out(F_N)} \phi K$.
By Proposition~\ref{prop:KN}, the map $J_{PS}: \cvv\to Curr(F_N)$
is a filling tree-current morphism.
The conclusion of the theorem now follows from Corollary~\ref{cor:main}. 
\end{proof}

\subsection{Uniform currents and generic stretching factors}

Kapovich and Nagnibeda also provide reasonably explicit description of $\mu_T$ in terms of its weights on the ``cylinder subsets'' of $\partial^2 F_N$.
The details of that description are not immediately relevant for the present paper. However, in the case where $T\in\cvv$ and where $T/F_N$ is a regular metric graph (that is, a regular graph where all edges have the same length), one can give a more precise description of $\mu_T$ as a ``uniform current'' corresponding to $T$ and relate $\mu_T$ to the exit measure of the simple non-backtracking random walk on $T$. We briefly recall here the description of uniform currents for the standard $N$-roses, that is for points of $\cvv$ corresponding to free bases of $F_N$.

Let $A=\{a_1,\dots, a_N\}$ be a free basis of $F_N$. Let $R_N$ be the graph given by a wedge of $N$ loop-edges $e_1,\dots, e_N$ at a vertex $x_0$. By identifying $e_i$ with $a_i\in F_N$ we get an identification of $\alpha_A: F_N \overset{\cong}{\longrightarrow} \pi_1(R_N,x_0)$, that is, a chart
on $F_N$. We give each edge of $R_N$ length $1/N$, so that $R_N$ becomes a metric graph of volume $1$. Then the universal cover $T_A:=\widetilde R_N$ is an $\mathbb R$-tree, which can be thought of as the Cayley graph of $F_N$ with respect to $A$, but where all edges have length $1/N$. The group $F_N$ has a natural 
free and discrete isometric 
left 
action on $T_A$ by covering transformations, with $T_A/F_N=R_N$. Thus $T_A$ is a point of $\cvv$.

The \emph{uniform current} $\nu_A$ on $F_N$ corresponding to $A$ is defined explicitly by its weights. Namely, for every non-trivial freely reduced word $v$ over $A^{\pm 1}$ we have
\[
\langle v, \nu_A\rangle_{\alpha_A}= \frac{1}{N(2N-1)^{|v|-1}}.
\]
One can check that this assignment of weights does define a geodesic current and that $\langle T_A, \nu_A\rangle=1$.
Moreover, in this case we also have:

\begin{prop}\label{prop:nu_A}
Let $N\ge 2$ and let $A$ be a free basis of $F_N$. Then $\mu_{T_A}=\nu_A$, that is, the Patterson-Sullivan current corresponding to $T_A$ is exactly the uniform current $\nu_A$.
\end{prop}

The above fact is not explicitly stated in \cite{KN} but it easily follows from the explicit formulas for the weights for Patterson-Sullivan currents obtained in \cite{KN}.
Alternatively, one knows, for example by the results of \cite{Coor,Ly} that for $T_A$ the uniform visibility measure $m_A$ on $\partial F_N=\partial T_A$ is a Patterson-Sullivan measure for $T_A$. Since $\nu_A\in \Curr(F_N)$ is in the measure class of $m_A\times m_A$ and since $\langle T_A,\nu_A\rangle=1$, it follows from the definition of the Patterson-Sullivan current that $\mu_{T_A}=\nu_A$.
Note that 
for any other $S\in\cvn$ the intersection number $\langle S,\nu_A\rangle$ measures the distortion of a ``long random geodesic'' in $T_A$ with respect to $S$.

Recall that in the Introduction, given a free basis $A$ of $F_N$, $S\in\cvnbar$ and $\phi\in\Out(F_N)$, we defined the generic stretching factors $\lambda_A(S)$ and $\lambda_A(\phi)$.

\begin{lem}\label{lem:1/N}
For any free basis $A$ of $F_N$ and any $S\in \cvnbar$
we have:
\[\lambda_A(S)\le 
\frac{1}{N} \Lambda(T_A,S).\]
\end{lem}

\begin{proof}

Since all edges in $T_A$ have length $1/N$, for every $w\in F_N$ we have  $||w||_A=N||w||_{T_A}$.
Then for a random trajectory $\xi=y_1y_2\dots y_n\dots$ of the simple non-backtracking random walk on $F_N$ with respect to $A$ we have:
\begin{gather*}
\lambda_A(S)=\lim_{n\to\infty} \frac{||y_1\dots y_n||_S}{||y_1\dots y_n||_A}=\lim_{n\to\infty} \frac{||y_1\dots y_n||_S}{
N||y_1\dots y_n||_{T_A}}=\\
\frac{1}{N}
\lim_{n\to\infty} \frac{||y_1\dots y_n||_S}{||y_1\dots y_n||_{T_A}
}\le 
\frac{1}{N}
\sup_{w\ne 1} \frac{||w||_S}{||w||_{T_A}}=
\frac{1}{N}
\Lambda(T_A,S).
\end{gather*}
\end{proof}

A key fact about generic stretching factors, originally established in \cite[Proposition 9.1]{Ka2} in slightly more limited context, is:

\begin{prop}\label{prop:random}
Let $A$ be a free basis of $F_N$ (where $N\ge 2)$ and let $S\in \cvnbar$.
Then
\[
\langle S, \nu_A\rangle=\lambda_A(S).
\]
\end{prop}

\begin{proof}
By \cite[Proposition 7.3]{Ka2}, for a.e. trajectory $\xi=y_1y_2\dots y_n \dots$ of the simple non-backtracking random walk on $F_N$ with respect to $A$, we have 
\[
\lim_{n\to\infty} \frac{1}{n}\eta_{y_1\dots y_n} =\nu_A.
\]
Therefore, by Proposition~\ref{prop:iform}, for any $S\in\cvnbar$ we have
\[
\langle S, \nu_A\rangle=\lim_{n\to\infty} \frac{1}{n}\langle S, \eta_{y_1\dots y_n} \rangle=\lim_{n\to\infty} \frac{||y_1\dots y_n||_S}{n}=\lambda_A(S) 
\]
\end{proof}

\begin{rem}\label{rem:>0}
Since the current $\nu_A$ has full support and therefore $\nu_A$ is filling, Proposition~\ref{prop:random} implies that for every $S\in \cvnbar$ we have $\lambda_A(S)>0$. (From the definition of $\lambda_A(S)$ one only knows that $\lambda_A(S)\ge 0$ and it is not a priori obvious, that the case $\lambda_A(S)=0$ cannot occur.)
\end{rem}

We can now obtain Corollary~\ref{cor:gen1} from the Introduction:

\begin{thm}\label{thm:gen1'}
Let $N\ge 2$. Then there exists a constant $\delta=\delta(N)\in (0,1)$ with the following property:

For any free basis $A$ of $F_N$ and any $S\in \cvnbar$ we have
\[
0<\delta\le \frac{\lambda_A(S)}{\Lambda(T_A,S)}\le \frac{1}{N}.
\]

\end{thm}
\begin{proof}
Let $A$ be a free basis of $F_N$ and let $S\in\cvnbar$ be arbitrary. By Lemma~\ref{lem:1/N}, we have
 $\frac{\lambda_A(S)}{\Lambda(T_A,S)}\le \frac{1}{N}$.

Let $\delta=\delta_1(\epsilon,N)>0$ be the constant provided by Theorem~\ref{thm:PS}. By decreasing this constant if necessary, we can always assume that $0<\delta_1<1$. Note that the length of the shortest essential circuit in $T_A$ is equal to 
$1/N$.

Since $0<\epsilon\le 1/N$, it follows that that $T_A\in\cvve$.
Since $\mu_{T_A}=\nu_A$ and $\langle S, \nu_A\rangle=\lambda_A(S)$,
by Theorem~\ref{thm:PS} we have
\[
0<\delta_1\le \frac{\langle S, \mu_{T_A}\rangle}{\Lambda(T_A,S)}=\frac{\langle S, \nu_{A}\rangle}{\Lambda(T_A,S)}=\frac{\lambda_A(S)}{\Lambda(T_A,S)}\le \frac{1}{N},
\]
as required.
\end{proof}

\section{Extremal, generic and algebraic stretching factors for free group automorphisms}

%We are particularly interested in relationships between generic stretching factors and extremal stretching factors in the context of Cayley trees of $F_N$ and of elements of $\Out(F_N)$. Note that if $A$ is a free basis of $A$ then $NT_A\in\cvn$ is the standard Cayley graph of $F_N$ with respect to $A$, where all edges have length $1$. 

%If $\phi\in\Out(F_N)$ and $w\in F_N$, then, since $\phi$ is an outer automorphism, it acts on the conjugacy classes of elements of $F_N$ (rather than on elements of $F_N$).
%By convention, for $\phi\in \Out(F_N)$ and $w\in F_N$, if $\phi(w)$ appears in an expression that depends only on the conjugacy class $\phi([w])$,  we will use $\phi(w)$ to mean any representative of that conjugacy class. 

%If $\phi\in\Out(F_N)$ and $w\in F_N$, then, since $\phi$ is an outer automorphism, the element $\phi(w)\in F_N$ is  only well-defined up-to conjugation in $F_N$. However, if $T\in \cvnbar$, the translation length $||g||_T$ only depends on the conjugacy class of $g$ in $F_N$. Therefore for $\phi\in\Out(F_N)$, $w\in F_N$ and $T\in\cvnbar$ the quantity $||\phi(w)||_T$ is well defined. By convention, in such cases we will interpret $||\phi(w)||_T$ to mean $||\Phi(w)||_T$ where $\Phi\in\Aut(F_N)$ is any representative of $\phi$ in $\Out(F_N)$.

We recall the notions of extremal and generic stretching factors from Definition~\ref{gen-stretch} in the Introduction:

\begin{defn}[Extremal and generic stretching factors of automorphisms]
\label{defn:gen-stretch}
Let $A$ be a free basis of $F_N$ and let $\phi\in\Out(F_N)$.

Denote
\[
\Lambda_A(\phi):=\Lambda(T_A, T_A\phi)=\sup_{w\ne 1} \frac{||\phi(w)||_A}{||w||_A}=e^{d_L(T_A,T_A\phi)}
\]
and refer to $\Lambda_A(\phi)$ as the \emph{extremal stretching factor} for $\phi$ with respect to $A$.

Also, denote 
$\lambda_A(\phi):=\lambda_A(NT_A\phi) = N\lambda_A(T_A\phi)$.

Thus for a.e.   trajectory $\xi=y_1\dots y_n\dots$ of the simple non-backtracking random walk on $F_N$ with respect to $A$ we have
\[
\lambda_A(\phi)=\lim_{n\to\infty} \frac{||\phi(y_1y_2\dots y_n)||_A}{n}=\lim_{n\to\infty} \frac{||\phi(y_1y_2\dots y_n)||_A}{||y_1y_2\dots y_n||_A}.
\]
We call $\lambda_A(\phi)$ the \emph{generic stretching factor} of $\phi$ with respect to $A$.
\end{defn}

First, we obtain, in a slightly restated form, Theorem~\ref{cor:gen2} from the Introduction:

\begin{thm}\label{cor:gen2'}
For every $N\ge 2$ there exists $0<\tau_N\le 1$ such that if $A$ is a free basis of $F_N$ and $\phi\in \Out(F_N)$ then
\[
0<\tau_N\le \frac{\lambda_A(\phi)}{\Lambda_A(\phi)} \le 1.
\]
\end{thm} 
\begin{proof}
Let $A$ be a free basis of $F_N$.  Recall that, by definition, for $\phi\in\Out(F_N)$ we have $\lambda_A(\phi)=N\lambda_A(T_A\phi)$ and $\Lambda_A(\phi)=\Lambda(T_A,T_A\phi)$. 
Therefore, by Lemma~\ref{lem:1/N}, we have  $\lambda_A(\phi)\le \Lambda_A(\phi)$, so that $ \frac{\lambda_A(\phi)}{\Lambda_A(\phi)} \le 1$. 
Since for any $\phi\in\Out(F_N)$ we have $T_A, T_A\phi\in\cvve$ with $\epsilon=1/N$, the statement of the theorem now follows directly from Theorem~\ref{thm:gen1'}.
\end{proof}

For two sequences $x_n>0, y_n>0$ (where $n\ge 1$) we say that $x_n$
\emph{grows like} $y_n$, if there exist
$0<c<c'<\infty$ such that for every $n\ge 1$ we have $c\le \frac{x_n}{y_n}\le c'$.

We now obtain Corollary~\ref{cor:gen3} from the Introduction:
\begin{cor}\label{cor:gen3'}
Let $N\ge 2$ and $F_N=F(a_1,\dots, a_n)$ with $A=\{a_1,\dots, a_N\}$.
There exists $D=D(N)\ge 1$ such that for every $\phi\in \Out(F_N)$ we
have
\[
\frac{1}{D} \log \lambda_A(\phi) \le \log \lambda_A(\phi^{-1})\le  D
\log \lambda_A(\phi) .
\]
\end{cor}
\begin{proof}
It follows from a result of Algom-Kfir and Bestvina~\cite[Theorem
24]{AKB} that there exists $D'=D'(N)\ge 1$ such that for every $\phi\in \Out(F_N)$ we
have
\[
\frac{1}{D'} d_L(T_A, T_A\phi) \le d_L(T_A\phi,T_A)\le  D' d_L(T_A, T_A\phi).
\]
Note that $d_L(T_A, T_A\phi)=\log \Lambda(T_A,T_A\phi)=\log
\Lambda_A(\phi)$ and that \[d_L(T_A\phi,T_A)=d_L(T_A, T_A\phi^{-1})=\log
\Lambda(T_A,T_A\phi^{-1})=\log \Lambda_A(\phi^{-1}).\] 

Theorem~\ref{cor:gen2'}  now implies that there exists $D''=D''(N)\ge 1$ such that 
for every $\phi\in \Out(F_N)$ we
have
\[
\frac{1}{D''} \log \lambda_A(\phi) -D'' \le \log \lambda_A(\phi^{-1})\le  D''
\log \lambda_A(\phi) +D''. \tag{$\ast\ast$}
\]

It was proved in \cite{Fra,KL3} (and also follows from Theorem~\ref{cor:gen2'}) that the set $\Omega_N:=\{\lambda_A(\phi) \mid\phi\in \Out(F_N)\}$ is a discrete subset of $[1,\infty)$.
It was established in \cite{KKS} that for $\lambda_A(\phi)=1$ if and only if $\phi$ is a \emph{permutational automorphism} with respect to $A$, that is, if and only if, after a possible composition with an inner automorphism, $\phi$ is induced by a permutation of $A$, with possibly inverting some elements of $A$. Note that $\phi$ is permutational with respect to $A$ if and only if  $\phi^{-1}$ is  permutational with respect to $A$, so that for $\phi\in Out(F_N)$  $\lambda_A(\phi^{-1})=1$ if and only if $\lambda_A(\phi)=1$.  It was also proved in \cite{KKS} that the minimum of $\lambda_A(\phi)$, taken over all non-permutational $\phi$, is equal to $1+\frac{2N-3}{2N^2-N}$.  
Therefore $(\ast\ast)$ 
implies that there exists $D=D(N)\ge 1$ such that for every non-permutational $\phi\in\Out(F_N)$ we have
\[
\frac{1}{D} \log \lambda_A(\phi) \le \log \lambda_A(\phi^{-1})\le  D
\log \lambda_A(\phi)  \tag{$\clubsuit$}
\]
If $\phi$ is permutational, then so is $\phi^{-1}$. In this case we have $\log\lambda_A(\phi^{-1})=\log\lambda_A(\phi)=0$ and $(\clubsuit)$ holds as well.
Thus $(\clubsuit)$ holds for every $\phi\in\Out(F_N)$, which completes the proof.
\end{proof}

Recall that for $\phi\in\Out(F_N)$ the 
\emph{algebraic stretching factor} $\lambda(\phi)$ is defined as
\[
\lambda(\phi)=\sup_{w\in F_N, w\ne 1}\lim_{n\to\infty} \sqrt[n]{||\phi^n(w)||_S}\]
where $S\in \cvn$ is an arbitrary base-point.  As noted earlier, this definition of $\lambda(\phi)$ does not depend on the choice of $S\in\cvn$. 
The algebraic stretching factor $\lambda(\phi)$ can be read-off from any relative train-track representative $f: \Gamma \to \Gamma$ of $\phi$ as the maximum of the Perron-Frobenius eigenvalues for any of the canonical irreducible diagonal blocks of the (non-negative) transition matrix $M(f)$.

Corollary~\ref{cor:power} below describes, given $\phi\in \Out(F_N)$, the asymptotics of $\Lambda(S,S\phi^n)$ as $n$ tends to infinity (where $S\in\cvnbar$ is an arbitrary point, the choice of which does not affect this asymptotics).
The statement of Corollary~\ref{cor:power} is probably known to the experts. Since the proof is not yet available in the literature, and since we need Corollary~\ref{cor:power} 
for the applications in this paper, we include the proof here.

\begin{prop}\label{prop:power}
Let $\phi\in\Out(F_N)$.
\begin{enumerate}
\item Let $q\ge 1$ and let $\alpha=\phi^q$ admit an improved relative
  train-track (in the sense of \cite{BFH00}) representative
$f:\Gamma\to\Gamma$. Put $\lambda:=1$ if $\alpha$ is polynomially growing
(that is, if $f$ has no exponentially growing strata) and otherwise
let $\lambda>1$ be
the largest Perron-Frobenius eigenvalue of the exponentially growing
strata of $f:\Gamma\to\Gamma$.

Then there exist and integer $m\ge 0$ such that for every $S\in \cvn$ there are some constants
constants $0<C_1\le C_2<\infty$ such that for every $n\ge 1$
\[
C_1\lambda^{n/q} n^m\le \Lambda(S,S\phi^n) \le C_2 \lambda^{n/q} n^m.
\]

\item If $\phi$ admits a train-track representative
$f:\Gamma\to\Gamma$ with an irreducible transition matrix  and
with the  Perron-Frobenius eigenvalue $\lambda>1$, then for every $S\in\cvn$ there exist $0<C_1\le
C_2<\infty$ such that for every $n\ge 1$
\[
C_1\lambda^n\le \Lambda(S, S\phi^n) \le C_2 \lambda^n.
\]
\end{enumerate}
\end{prop}

\begin{proof}

(1) Let $T\in \cvv$ be the point corresponding to the improved relative
train-track  $f:\Gamma\to\Gamma$, where all edges of $\Gamma$ are
given equal length. Put $L=\{1\}$ if $f$ has no exponentially growing
strata. Otherwise let $\lambda_1 \ge \dots \ge \lambda_k>1$ be all the Perron-Frobenius eigenvalues of the exponentially growing
strata of $f$ and put $L=\{\lambda_1,\dots, \lambda_k, 1\}$. Finally
put $\lambda=\max L$. Thus $\lambda\ge 1$ and $\lambda=1$ if and only
if $f$ has no exponential strata.

A result of Levitt~\cite[Theorem 6.2]{Levitt} shows that there is a finite subset $M$ of $\Z_{\ge 0}$ such that for every
non-trivial $w\in F_N$ there is some $(\lambda',m')\in L\times M$ such that the sequence $||\alpha^n(w)||_T$ grows like
$(\lambda')^n n^{m'}$.  Moreover, there exists some element $1\ne w_0\in F_N$  such that
$||\alpha^n(w_0)||_T$ grows as $\lambda^n n^m$ and such that if some other $w\ne 1$ has $||\alpha^n(w)||_T$ growing as $\lambda^nn^{m'}$ then $m'\le m$.

Let $D=C_\Delta$ be the finite subset of $F_N$ as
in Remark~\ref{rem:candidates}, where $\Delta$ is the open simplex in
$\cvv$ containing $T$. Therefore for every $n\ge 1$ we have 
$\Lambda(T,T\phi^n)=\max_{w\in D}
\frac{||\alpha^n(w)||_T}{||w||_T}$. Moreover, 
through 
replacing $D$ by
$D\cup\{w_0\}$ we can assume that $w_0\in D$.

It follows that $\Lambda(T, T\alpha^n)=\max_{w\in D}
\frac{||\alpha^n(w)||_T}{||w||_T}$ grows like $\lambda^n n^m$.

Now let $n\ge 1$ and write $n=qn_1+r$ where $n_1\ge 0$ and $0\le
r\le q-1$ are integers. As we have seen, 
$\Lambda(T,T\alpha^{n_1})=\max_{w\in D}
\frac{||\phi^{n_1}(w)||_T}{||w||_T}$ grows like $\lambda^{n_1}n_1^m$. Since
$0\le r\le q-1$, applying $\phi^r$ distorts $||.||_T$ by a bounded
multiplicative amount. Therefore $\Lambda(T,T\phi^n)=\max_{w\in D}
\frac{||\phi^n(w)||_T}{||w||_T}$ grows as $\lambda^{n/q}(n/q)^m$, that
is, as $\lambda^{n/q}n^m$.

Since $T$ and $S$ are $F_N$-equivariantly quasi-isometric, it
follows that $\Lambda(S, S\phi^n)=\Lambda_A(\phi^n)$ also grows like
$\lambda^{n/q} n^m$, and the conclusion of part (1) of the proposition follows.

(2) The proof of part (2) is known (e.g. see Theorem~8.1 in \cite{FM11}) and is simpler than the proof of part (1), and we leave the details to the
reader. The key point is that in this case for every non-trivial $w\in
F_N$, such that the conjugacy class of $w$ is not $\phi$-periodic, the
sequence $||\phi^n(w)||_S$ grows like $\lambda^n$.
\end{proof}

\begin{cor}\label{cor:power}
Let $\phi\in\Out(F_N)$, let $S\in\cvn$ and let 
$\lambda(\phi)$ be the algebraic stretching factor of $\phi$.

Then there is an integer $m\ge 0$ such that for every $S\in\cvn$ there are some $C_1,C_2>0$ such that

\[
C_1\,\lambda(\phi)^n \, n^m\le \Lambda(S, S\phi^n) \le C_2 \,\lambda(\phi)^n \, n^m
\]
for all $n\ge 1$.
\end{cor}

\begin{proof}
It is known~\cite{BFH00} that some positive power $\alpha=\phi^q$ of $\phi$ admits
an improved relative train track representative. 

In this case we have $\lambda(\alpha)=\lambda(\phi^q)=\lambda(\phi)^q$, so that $[\lambda(\alpha)]^{1/q}=\lambda(\phi)$.
The conclusion of the corollary now follows directly from part (1) of Proposition~\ref{prop:power}.
\end{proof}

Now Corollary~\ref{cor:power} 
(applied to $S= T_A$, which gives $\Lambda(S, S \phi^n) = \Lambda_A(\phi^n)$)
and Theorem~\ref{cor:gen2'} directly imply Theorem~\ref{thm:p} from the Introduction:

\begin{thm}\label{thm:p'}
Let $A$ be a free basis of $F_N$ and let $\phi\in\Out(F_N)$ and let $
\lambda(\phi)$ be the algebraic stretching factor of $\phi$. Then there exist constants $c_1,c_2>0$ and an integer $m\ge 0$ such that for every $n\ge 1$ we have

\[
c_1  \, \lambda(\phi)^n \,  n^m\le \lambda_A(\phi^n) \le c_2  \,\lambda(\phi)^n \, n^m.
\]
Moreover, if $\phi$ admits an expanding train-track representative with
an irreducible transition matrix (e.g. if $\phi$ is fully
irreducible), then $m=0$ and $\lambda(\phi)>1$.
\qed
\end{thm} 

\begin{exmp}
To demonstrate that the case $\lambda>1,m>0$ in Theorem~\ref{thm:p'} can indeed occur, we consider an example explained on p.  1138 in \cite{Levitt}. Let $N=4$ and $F_4=F(A)$ with $A=\{a_1,b_1,a_2,b_2\}$. Let an automorphism $\phi:F(A)\to F(A)$ be given by
\[
\phi(a_1)=a_1b_1,\  \phi(b_1)=a_1, \ \phi(a_2)=a_2b_1a_1,\ \phi(b_2)=a_2.
\]
For the $A$-rose $R_A$ the map $f:R_A\to R_A$, given by the same formula as $\phi$, is both a global train-track and a 2-strata relative train-track representative for $\phi$. The bottom stratum is $\{a_1,b_1\}$ and the top stratum is $\{a_2,b_2\}$. The transition matrices for both strata are the same and are equal to $B=\begin{bmatrix} 1& 1\\ 1 & 0\end{bmatrix}$, which has the Perron-Frobenius eigenvalue $\lambda=\frac{1+\sqrt{5}}{2}$. The transition matrix for $f$ has the form $M=\begin{bmatrix} B& 0\\ C & B\end{bmatrix}$ where $C=\begin{bmatrix} 1& 0\\ 0 & 0\end{bmatrix}$. By iterating $M$ one can see that $||\phi^n(a_2)||_A$ grows like $n\lambda^n$. One can then show that in this case $\Lambda_A(\phi^n)$  also grows as $n\lambda^n$.  Therefore, by Theorem~\ref{cor:gen2'}, $\lambda_A(\phi^n)$ grows as $n\lambda^n$ as well. 
\end{exmp}

\section{Other examples of filling tree-current morphisms}

The Patterson-Sullivan map $J_{PS}: \cvv\to \Curr(F_N)$, $T\mapsto \mu_T$,  is just one, albeit natural and useful, example of a filling tree-current morphism.
There are many other filling tree-current morphisms $J:\cvv\to Curr(F_N)$, and Corollary~\ref{cor:main} is applicable to all such $J$.
We indicate here some sources of such $J$, following the approach of Reiner Martin~\cite{Martin}.
The main idea is that if $t \mapsto\rho(t)>0$  is a
monotone decreasing  continuous function which
approaches $0$ as $t\to\infty$ ``sufficiently quickly'', then  
\[
J_\rho: \cvv\to\Curr(F_N), \, T\mapsto \sum_{[w] \neq [1]} \rho(||w||_T) \eta_w
\]
is a filling tree-current morphism.  

The summation here can be taken either over all non-trivial conjugacy classes $[w]$ of elements of $F_N$ (or over an $\Out(F_N)$-invariant set of such conjugacy classes although in the latter case one has to take additional care to ensure that 
the current 
$J_\rho(T)$ is 
filling).

Let us 
first 
observe that such a function $J_\rho$ is, 
by its construction, 
always $\Out(F_N)$-equivariant:  For any $T\in\cvv$ and $\phi\in\Out(F_N)$ we have
\[
\phi(J_\rho(T))=\sum_{[w] \neq [1]} \rho(||w||_T) \phi(\eta_w) =\sum_{[w] \neq [1]} \rho(||w||_T) \eta_{\phi(w)}
\]
and
%\marginpar{\tiny ML:  I think this looks now a bit nicer...}
\begin{gather*}
J_\rho(\phi T)= \sum_{[w] \neq [1]} \rho(||w||_{\phi T}) \eta_w= \sum_{[w] \neq [1]} \rho(||\phi^{-1}(w)||_T) \eta_w\\
\underset{\text{ with } u=\phi^{-1}(w)}{=} \,\,\,\sum_{[u] \neq [1]} \rho(||u||_T) \eta_{\phi(u)}=\phi(J_\rho(T)),
\end{gather*}
so that $J_\rho$ is indeed $\Out(F_N)$-equivariant.

We provide here a representative result of the kind described above:

\begin{prop}
The function 
\[
J:  \cvv \to \Curr(F_N), \, T\mapsto\sum_{[w] \neq [1]} e^{-e^{||w||_T}} \eta_w
\] 
where
the sum is taken over all non-trivial root-free conjugacy classes $[w]$ of
elements of $F_N$, is an injective filling tree-current morphism.
\end{prop}

\begin{proof}
Fix a free basis $A$ of $F_N$ and let $T_A\in\cvv$ be the
Cayley graph of $F_N$ with respect to $A$, where all edges in $T_A$
have length $1/N$. For $w\in F_N$ denote by $||w||_A$ the cyclically
reduced length of $w$ over $A^{\pm 1}$. Thus $||w||_A=N||w||_{T_A}$.
We let $R_A=T_A/F_N$ be the quotient metric graph, which is a wedge of
$N$ loop-edges of length $1/N$ corresponding to elements of $A$. Let
$\alpha_A:F_N\to\pi_1(R_A)$ be the associated 
%marking.
chart.

Let $T\in \cvv$ be arbitrary and let $U$ be a compact neighborhood of
$T$ in $\cvv$. There exists a constant $C\ge 1$ such that for every $w\in F_N$
and every $T'\in U$ we have $||w||_{T'}/C\le ||w||_A\le C||w||_{T'}$.
Note that for $n\ge 1$ the number of conjugacy classes $[w]$ with $||w||_A\le n$ is
$\le (2N)^{n}$.

To show that for each $T'\in U$ $J(T')$ is a geodesic current we only need to verify that
$J(T')$ takes finite values on all the two-sided cylinder sets in
$\partial^2 F_N$  determined by the 
chart 
$\alpha_A$. Since every
cylinder is contained in a cylinder determined by a single edge, it
suffices to show that for every oriented edge $e$ of $R_A$ we have
$\langle e, J(T')\rangle_{\alpha_A}<\infty$.

Let $T'\in U$ and let $e$ be an edge of $R_A$. For every integer $n\ge 1$ 
set
\[
b_n(e,T'):=\sum_{0.9n\le ||[w]||_A\le 1.1n} e^{-e^{||w||_{T'}}} \langle
e, \eta_w\rangle_{\alpha_A}.
\]

Then $\langle e, J(T')\rangle_{\alpha_A}\le \sum_{n=1}^{\infty}
b_n(e,T')$. The weight  $\langle e, \eta_w\rangle_{\alpha_A}$ is equal
to $1/N$ times the number of occurrences of $e^{\pm 1}$ in the
cyclically reduced circuit $\gamma_w$ in $R_A$ representing
$[w]$. Hence  $\langle e, \eta_w\rangle_{\alpha_A}\le
\frac{1}{N}||w||_A$. Since $T'\in U$, we have $||w||_{T'}\ge
  ||w||_A/C$.
Hence for every $n\ge 1$ and $T'\in U$ we have

\begin{gather*}
b_n(e,T')=\sum_{0.9n\le ||[w]||_A\le 1.1n} e^{-e^{||w||_{T'}}} \langle
e, \eta_w\rangle_{\alpha_A}\le\\ \frac{1}{N}\sum_{0.9n\le ||[w]||_A\le 1.1n}
e^{-e^{||w||_{A}/C}} ||w||_A \le 
\frac{1}{N}\sum_{0.9n\le ||[w]||_A\le 1.1n}
e^{-e^{0.9n/C}} 1.1n \le\\
\frac{1.1n}{N}
e^{-e^{0.9n/C}}(2N)^{1.1n}=\frac{1.1n}{N}
e^{-e^{0.9n/C}}e^{1.1n\log(2N)}=\\
\frac{1.1n}{N}e^{1.1n\log(2N)-e^{0.9n/C}}
\end{gather*}
From here we see that 
\[
\langle e, J(T')\rangle_{\alpha_A}\le \sum_{n=1}^{\infty}
b_n(e,T') \le C_1
\]
where $C_1=C_1(U)<\infty$ is some constant depending only on $U$.

Thus for every $T'\in U$ $J(T')$ is indeed a geodesic current on
$F_N$, and, in particular, $J(T)\in \Curr(F_N)$.

Note that the current $J(T)$ has full support. Indeed, for every non-trivial
freely reduced word $v$ over $A^{\pm 1}$ there exists a root-free cyclically
reduced word $w$ over $A^{\pm 1}$ containing $v$ as a subword. Then
$\langle v,\eta_w\rangle_{\alpha_A} >0$ and hence, from the definition
of $J(T)$, we see that $\langle v,J(T)\rangle_{\alpha_A} >0$. Thus
indeed $J(T)$ has full support and therefore, by a result of
Kapovich-Lustig~\cite{KL3}, the current $J(T)$ is filling.

Since an automorphism of $F_N$ permutes the set of all root-free
non-trivial conjugacy classes in $F_N$, it follows from the definition of $J$ that for every $T\in\cvv$
and every $\phi\in\Out(F_N)$ we have $J(\phi T)=\phi J(T)$.

Thus we have constructed an $\Out(F_N)$-equivariant map $J:\cvv\to
\Curr_{fill}(F_N)$.

We next observe that the map $J$ is continuous. The proof of continuity of $J$ is similar to the proof that $J(T)$ is
a current.  Let $T\in\cvv$, $U$ be a
compact neighborhood of $T$ in $\cvv$ and let $v$ be a non-trivial
freely reduced word over $A^{\pm 1}$.
Then for every $T'\in U$  we have
\[
\langle v, T'\rangle_{\alpha_A}=\sum_{[w]} \langle v, e^{-e^{||w||_{T'}}} \eta_w \rangle_{\alpha_A}=\sum_{[w]}  e^{-e^{||w||_{T'}}}  \langle v, w\rangle_{\alpha_A}
\]

One can then show, by an argument similar to that used above, that there exist positive constants $M_w>0$ (also depending on $U$ and $v$ but independent of $T'\in U$) such that for every $T'\in U$ we have $e^{-e^{||w||_{T'}}}  \langle v, w\rangle_{\alpha_A}\le M_w$ and that $\sum_{[w]} M_w<\infty$. By the Weirstrass $M$-test, it follows that series $\sum_{[w]}  e^{-e^{||w||_{T'}}}  \langle v, w\rangle_{\alpha_A}$, viewed as the sum of a functions on $U$, converges uniformly on $U$ and that its sum $\langle v, T'\rangle_{\alpha_A}$ is a continuous function on $U$.

Since $v$ was arbitrary, the explicit description of the
topology on $\Curr(F_N)$ (see~\cite{Ka2}) implies that $J$ is a continuous
function on $\cvv$, as required.

It remains to show that $J$ is injective. Fix an enumeration, without
repetitions, $w_1,w_2,\dots, $ of representatives of all the non-trivial root-free
conjugacy classes in $F_N$. Thus for every root-free non-trivial $w\in
F_N$ there exist unique distinct $m,n\ge 1$ such that $[w]=[w_m]$ and $[w^{-1}]=[w_n]$.

For every $i\ge 1$ 
set
$q_i=(w_i^{-\infty},w_i^\infty)\in \partial^2 F_N$ and 
set
$Q_i=\{q_i\}$. Note that for $i,j\ge 1$ we have $\eta_{w_j}(Q_i)=1$ if
$[w_i]=[w_j^{\pm 1}]$ and $\eta_{w_i}(Q_i)=0$ otherwise. Then, by
definition of $J$,  for every $T\in \cvv$ and $i\ge 1$
we have  $J(T)(Q_i)=2e^{-e^{||w_i||_T}}$. Since the function $t\mapsto
2e^{-e^t}$ is strictly monotone and thus injective, it follows that
knowing the current $J(T)$ we can recover $||w_i||_T$ for all $i\ge
1$. Hence we can recover the length function $||\cdot||_T:F_N\to\mathbb R$
and so we can also recover $T$ itself. Thus $J$ is injective,
as required.

\end{proof}

\section{Open problems}

As we have seen in Theorem~\ref{cor:gen2}, if $N\ge 2$, $A=\{a_1,\dots, a_N\}$ is a fixed free basis of $F_N=F(A)$, then for
\[
\rho_N=\inf_{\phi\in\Out(F_N)} \frac{\lambda_A(\phi)}{\Lambda_A(\phi)}
\]
we have $\rho_N>0$. In fact, one can show:

\begin{prop}\label{prop:limit}
We have $\displaystyle\lim_{N\to\infty} \rho_N=0$, and moreover, $\rho_N=O(\frac{1}{N})$, that is $\displaystyle\limsup_{N\to\infty}N\rho_N<\infty$.
\end{prop}
\begin{proof}
For $N\ge 2$ and $m\ge 1$  let $\phi_{N,m}:F(A)\to F(A)$ be given by $\phi_{N,m}(a_1)=a_1a_2^m$ and $\phi_{N,m}(a_i)=a_i$ for $2\le i\le N$.
It is not hard to see that $\Lambda_A(\phi_{N,m})=\sup_{w\ne 1} \frac{||\phi_{N,m}(w)||_A}{||w||_A}=m+1$.
For any freely reduced $w\in F(A)$ we have $||\phi_{N,m}(w)||_A\le (m+1)(a_1; w)_A +\sum_{i=2}^N (a_i; w)_A$, where  $(a_j;w)_A$ is the number of occurrences of $a_j^{\pm 1}$ in $w$.  On the other hand, if $w_n\in F(A)$ a ``long random'' freely reduced word of length $n$, then asymptotically we have $\frac{(a_i; w_n)_A}{n}
%\to_{n\to\infty} 
\overset{n \to \infty}{\longrightarrow}
\frac{1}{N}$ for $i=1,\dots, N$. Therefore
\[
\lambda_A(\phi_{N,m})\le \lim_{n\to \infty} \frac{(m+1)(a_1; w)_A +\sum_{i=2}^N (a_i; w)_A}{n}=(m+1)\frac{1}{N}+\frac{N-1}{N}=\frac{m}{N}+1.
\]
Hence
\[
\rho_N\le \frac{\lambda_A(\phi_{N,m})}{\Lambda_A(\phi_{N,m})}\le \frac{1+\frac{m}{N}}{m+1}.
\]
By taking $m=N$, we see that $\displaystyle \rho_N\le  \frac{2}{N +1}
\overset{n \to \infty}{\longrightarrow}
0$. Thus  $\displaystyle\lim_{N\to\infty} \rho_N=0$ and  $\displaystyle\limsup_{N\to\infty}N\rho_N<\infty$.
\end{proof}

Theorem~\ref{cor:gen2} and Proposition~\ref{prop:limit} naturally raise the following:
\begin{prob}\label{prop:1}
Are the values $\rho_N$ algorithmically computable in terms of $N$? What are the exact values of $\rho_N$ for small $N$, say for $N=2,3,4$? Is it true that $\rho_N\in \mathbb Q$?
What can be said about the precise asymptotics of $\rho_N$ as $N\to\infty$?  (Note that Proposition~\ref{prop:limit} shows that $\rho_N$ decays at least as fast as $1/N$.) 
\end{prob}

Theorem~\ref{thm:A} also motivates the definition of a new notion of a continuous symmetric and $\Out(F_N)$-invariant intersection number $I:\cvv\times\cvv\to \R_{>0}$, where for $T,S\in\cvv$ we define $I(T,S):=\langle S,\mu_T\rangle\langle T,\mu_S\rangle$.  The function $I(\cdot,\cdot)$ was  originally suggested to us by Arnaud Hilion as it appears to be relevant for attempting to define an analogue of the Weil-Petersson metric on $\cvv$. 

Since the Patterson-Sullivan currents are normalized so that $\langle T,\mu_T\rangle =1$, for $T=S$ we have $I(T,T)=1$.
\begin{prob}\label{prob:2} $ $

\noindent (a) Is it true that for every $T,S\in \cvv$ we have $I(T,S)\ge 1$~?

\smallskip
\noindent
(b)
Is it true that for $T,S\in \cvv$  we have $I(T,S)=1$ if and only if $T=S$~?

\end{prob}
It was shown in~\cite{KKS} that if $A$ is a free basis of $F_N$ and $\phi\in \Out(F_N)$ then $\lambda_A(\phi)\ge 1$ and that $\lambda_A(\phi)=1$ if and only if $T_A\phi=T_A$.
If $B$ is another free basis of $F_N$ and $\phi\in\Aut(F_N)$ is such that  $T_A\phi=T_B$, then $\langle T_B, \mu_{T_A}\rangle=\lambda_A(\phi)$ and $\langle T_A, \mu_{T_B}\rangle=\lambda_A(\phi^{-1})$. It follows that if $A, B$ are free bases of $F_N$ then $I(T_A,T_B)\ge 1$ and that $I(T_A,T_B)=1$ if and only if $T_A=T_B$. However, beyond this fact nothing appears to be known about the above question.

Recently  Pollicott and Sharp~\cite{PS}, using a different approach,  defined and studied a Weil-Petersson type metric on $\cvv$. It would be interesting to investigate the
relationship of their metric to the quantity $I(T,S)$ defined above.

\end{document}